\documentclass[regno]{amsart}
\usepackage[utf8]{inputenc}
\usepackage{mathrsfs}
\usepackage{mathtools}
\usepackage{amsmath}
\usepackage{amsthm}
\numberwithin{equation}{section}
\usepackage[left=2.5cm, right=2.5cm]{geometry}
\makeindex
\usepackage{multicol}
\usepackage{tikz,pgf}
\usepackage{tikz-cd}
\usepackage{float}
\usepackage[all,knot,arc,color,web]{xy}
\usepackage{graphicx}
\xyoption{curve}
\usepackage{bbold}
\usepackage{fancyhdr}
\theoremstyle{plain}
\newtheorem{theorem}{Theorem}[section]
\newtheorem{proposition}[theorem]{Proposition}
\newtheorem{corollary}[theorem]{Corollary}
\newtheorem{lemma}[theorem]{Lemma}
\newtheorem{definition}[theorem]{Definition}
\theoremstyle{definition}

\long\def\comment#1{}
\title{Higher order terms of Mather's $\beta$-function
for symplectic and outer billiards}
\author{Luca Baracco, Olga Bernardi, Alessandra Nardi}
\address{Dipartimento di Matematica Tullio Levi-Civita, Universit\`a di Padova, via Trieste 63, 35121 Padova, Italy}
\email{baracco@math.unipd.it, obern@math.unipd.it, nardi@math.unipd.it}
\begin{document}
\maketitle
\begin{abstract}
We compute explicitly the higher order terms of the formal Taylor expansion of Mather's $\beta$-function for symplectic and outer billiards in a strictly-convex planar domain $C$. In particular, we specify the third terms of the asymptotic expansions of the distance (in the sense of the symmetric difference metric) between $C$ and its best approximating inscribed or circumscribed polygons with at most $n$ vertices. We use tools from affine differential geometry.
\end{abstract}
\section{Introduction}
The aim of this paper is providing an explicit representation of higher order terms for the formal Taylor expansion of Mather's $\beta$-function (or minimal average action) for symplectic and outer billiards. In particular, we write these new coefficients in terms of the affine curvature and length of the boundary of the billiard table. In the specific case of symplectic billiards, the corresponding Mather's $\beta$-function is related to the maximal area of polygons inscribed in the billiard table. Conversely, for outer billiards, such a function corresponds to the minimal area of circumscribed polygons. These areas are special cases (i.e. for periodic billiard trajectories of winding number $=1$) of the corresponding marked area spectrum invariants for symplectic and outer billiards. In order to state our main theorem, we need some preliminaries. 

\subsection{Twist maps and Mather's $\beta$-function}


Let $\mathbb{S}^1 \times (a,b)$ be the annulus, where $\mathbb{S}^1 = \mathbb{R}/\mathbb{Z}$ and we allow that $a=-\infty$ and/or $b= +\infty$. Given a diffeomorphism 
$$\Phi: \mathbb{S}^1 \times (a,b) \rightarrow \mathbb{S}^1 \times (a,b),$$ 
we assume --if $a$ (resp. $b$) is finite-- that $\Phi$ extends continuously to $\mathbb{S}^1 \times \{a\}$ (resp. $\mathbb{S}^1 \times \{b\}$) by a rotation of fixed angle. Moreover, we denote by 
$$\phi: \mathbb{R} \times (a,b) \rightarrow \mathbb{R} \times (a,b)$$
a lift of $\Phi$ to the universal cover. Then $\phi$ is a diffeomorphism and $\phi(x+1,y) = \phi(x,y) + (1,0)$. The next definition of monotone twist map is well consolidated in literature, we refer e.g. to \cite{Si}[Page 2]. 
\begin{definition}
A monotone twist map $\phi: \mathbb{R} \times (a,b) \rightarrow \mathbb{R} \times (a,b)$, $(x_0,y_0) \mapsto (x_1,y_1)$
is a diffeomorphism satisfying:
\begin{enumerate}
\item $\phi(x_0+1,y_0) = \phi(x_0,y_0) + (1,0)$.
\item $\phi$ preserves orientations and the boundaries of $\mathbb{R} \times (a,b)$. 
\item $\phi$ extends to the boundaries by rotation: $\phi(x_0,a) = (x_0 + \rho_a,a)$ and $\phi(x_0,b) = (x_0 + \rho_b,b)$. 
\item $\phi$ satisfies a monotone twist condition, that is
\begin{equation} \label{TC}
    \frac{\partial x_1}{\partial y_0} > 0.
\end{equation}
\item $\phi$ is exact symplectic; this means that there exists a generating function $S \in C^2(\mathbb{R} \times \mathbb{R}; \mathbb{R})$ for $\phi$ such that
\begin{equation} \label{GF}
y_1dx_1 - y_0dx_0 = dS(x_0,x_1).
\end{equation}
\end{enumerate}
\end{definition}
\noindent Clearly, $S(x_0+1,x_1+1) = S(x_0,x_1)$ and, due to the twist condition, the domain of $S$ is the strip $\{(x_0,x_1): \ \rho_a + x_0 < x_1 < x_0 + \rho_b\}$. Moreover, equality (\ref{GF}) reads
\begin{equation} \label{ESSE}
\begin{cases} y_1=S_2(x_0,x_1)\\
y_0=-S_1(x_0,x_1)
\end{cases}
\end{equation}
and the twist condition (\ref{TC}) becomes $S_{12} < 0$. As a consequence of the monotone twist condition and (\ref{ESSE}), $\{(x_i,y_i)\}_{i \in \mathbb{Z}}$ is an orbit of $\phi$ if and only if $S_2(x_{i-1},x_i) = y_i = -S_1(x_i,x_{i+1})$ for all $i \in \mathbb{Z}$. Equivalently, the corresponding bi-infinite sequence $x := \{x_i\}_{i \in \mathbb{Z}}$ is a so-called critical configuration of the action functional: 
\begin{equation*}
\sum_{i\in \mathbb{Z}} S(x_i,x_{i+1}).
\end{equation*}
We say that a critical configuration $x$ of $\phi$ is minimal if every finite segment of $x$ minimizes the action functional with fixed end points (we refer to \cite{Si}[Page 7] for details). For a twist map $\phi$ generated by $S$, we finally introduce the rotation number and the Mather's $\beta$-function (or minimal average action).
\begin{definition} The rotation number of an orbit $\{(x_i,y_i)\}_{i \in \mathbb{Z}}$ of $\phi$ is 
$$\rho := \lim_{i\to\pm\infty}\frac{x_i}{i}$$
if such a limit exists. 
\end{definition}
\noindent In view of the celebrated Aubry-Mather theory (see e.g. \cite{Bang}), a monotone twist map possesses minimal orbits for every rotation number $\rho$ inside the so-called twist interval $(\rho_{a},\rho_{b})$. 
\begin{definition} \label{Mather beta}
The Mather's $\beta$-function of $\phi$ is $\beta: (\rho_a,\rho_b) \to \mathbb{R}$ with
$$\beta(\rho) := \lim_{N\to\infty}\frac{1}{2N}\sum_{i=-N}^{N-1} S(x_i,x_{i+1})$$
where $\lbrace x_i\rbrace_{i\in \mathbb{Z}}$ is any minimal configuration of $\phi$ with rotation number $\rho$.
\end{definition}
\noindent Clearly, all these facts remain true if we consider a monotone twist map on $\{(x_0,x_1): \ u_a(x_0) < x_1 < u_b(x_0)\}$. A relevant class of monotone twist maps are planar billiard maps. The study of such systems goes back to G.D. Birkhoff \cite{Bir}, who introduced the so-called Birkhoff billiard map, where the reflection rule is ``angle of incidence = angle of reflection''. In the setting of planar billiards, the rotation number of a periodic trajectory is the rational 
$$\frac{m}{n} = \frac{\text{winding number}}{\text{number of reflections}} \in (0, \frac{1}{2}\Big{]},$$ 
\noindent we refer to \cite{Si}[Page 40] for details. \\
A. Sorrentino in \cite{Sor} gave an explicit representation of the coefficients of (formal) Taylor expansion at zero of Mather’s $\beta$-function associated to Birkhoff billiards. More recently, J. Zhang in \cite{Zhang} got (locally) an explicit formula for this function via a Birkhoff normal form. Moreover, M. Bialy in \cite{BiEl} obtained an explicit formula for Mather's $\beta$-function for ellipses by using a non-standard generating function (involving the support function) of the billiard problem.



\subsection{Symplectic billiards} \label{SB}

Let $C$ be a strictly-convex planar domain with smooth boundary $\partial C$ and fixed orientation. Moreover, throughout the paper, we suppose that $\partial C$ has everywhere positive curvature. 
Since $C$ is  strictly-convex, for every point $x \in \partial C$ there exists a unique point $x^*$ such that 
$$T_x\partial C = T_{x^*}\partial C.$$  
We refer to
\begin{equation*} \label{PS}
\mathcal{P} := \{ (x, y) \in \partial C \times \partial C: \ x < y < x^*\}
\end{equation*}
    as the (open, positive) phase-space and we define the symplectic billiard map as follows (see \cite{ALB}[Page 5]):
$$\Phi: \mathcal{P} \rightarrow \mathcal{P}, \qquad (x,y) \mapsto (y,z)$$
where $z$ is the unique point satisfying 
$$z-x \in T_y \partial C.$$
We refer to Figure \ref{Figuraccia}. We notice that $\Phi$ is continuous and can be continuously extended to $\bar{\mathcal{P}}$ so that 
$$\Phi(x,x) = (x,x) \qquad \text{and} \qquad \Phi(x,x^*) = (x^*,x)$$
and therefore the twist interval is $(0,1/2)$. Moreover, see \cite{ALB}[Section 2] for exhaustive details, the symplectic billiard map turns out to be a twist map with generating function $S(x,y)=\omega(x,y)$ where $\omega$ is the area form
$$\omega: \mathcal{P} \to \mathbb{R}, \qquad (x,y) \mapsto \omega(x,y),$$
with dynamics given by
$$(y,z) = \Phi(z,y) \quad \text{iff} \quad \frac{\partial}{\partial y} \left[\omega(x,y) + \omega(y,z) \right] = 0.$$ 
We refer also to \cite{BaBe} for recent advances on symplectic billiards.
\begin{definition}
The marked area spectrum for the symplectic billiard is the map
$$\mathcal{MA}_s(C): \mathbb{Q} \cap{(0, \frac{1}{2}\Big{)}} \to \mathbb{R}$$
that associates to any $m/n$ in lowest terms the maximal area of the periodic trajectories having rotation number $m/n$. 
\end{definition}
\noindent We refer to \cite{Si}[Sections 3.1 and 3.2] for an exaustive treatment of the marked length spectrum and corresponding invariants for the Birkhoff billiard map. Clearly, periodic symplectic billiard maximal trajectories (with winding number $=1$) correspond to convex polygons realizing the maximal (inscribed) area. We call them best approximating polygons inscribed in $C$. More precisely, let $\mathcal{P}_n^i$ the set of all convex polygons with at most $n$ vertices that are inscribed in $C$. 
We define
\begin{equation}\label{7}
\delta(C, \mathcal{P}_n^i):=inf\lbrace \delta(C,P_n) : P_n\in\mathcal{P}_n^i\rbrace
\end{equation}
where $\delta(C,P_n)$ is the area of the complementary of $P_n$ in $C$. Then: 
\begin{equation} \label{beta-MA}
    \beta\left( \frac{1}{n}\right) = -\frac{2}{n} \mathcal{MA}_s(C) \left( \frac{1}{n}\right) = -\frac{2}{n} \left( Area(C) - \delta(C,P^i_n) \right).
\end{equation}
We underline that the sign minus in the above equality comes from the use of the generating function $-\omega(x,y)$; in fact --according to Definition \ref{Mather beta}-- the Mather's $\beta$-function is defined by using minimal, instead of maximal, trajectories. 
\begin{figure}
    \centering
    \includegraphics{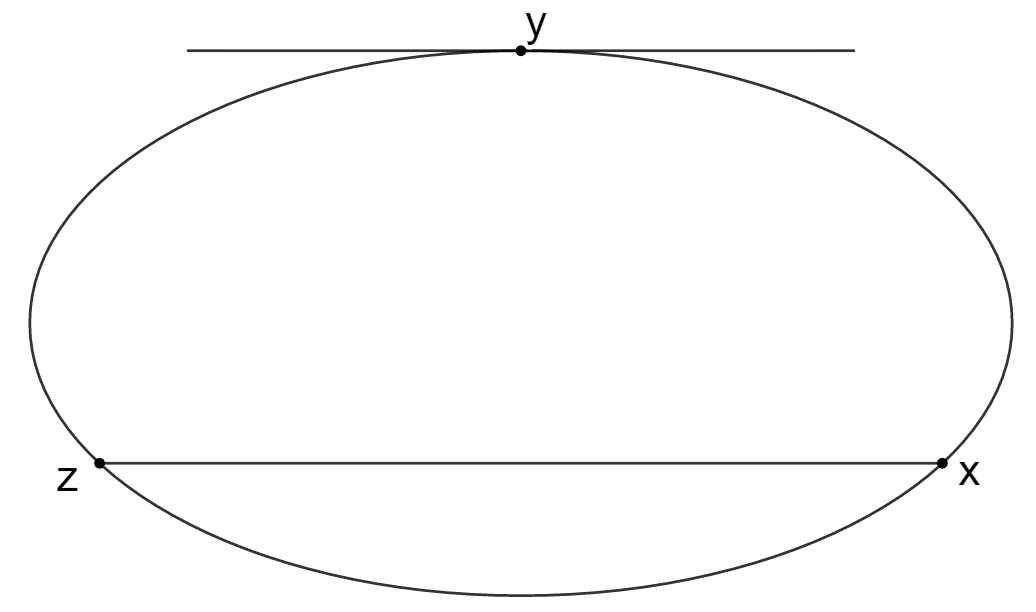}
    \caption{The symplectic billiard map reflection.}
    \label{Figuraccia}
\end{figure}

\subsection{Outer billiards} \label{OB}

\noindent For the same planar domain $C$ as in Section \ref{SB}, we finally briefly introduce the outer billiard map, which is defined on the exterior of $\partial C$ as follows. A point $A$ is mapped to $\Phi(A)$ iff the segment joining $A$ and $\Phi(A)$ is tangent to $\partial C$ exactly at its middle point and has positive orientation at the tangent point. We refer to Figure \ref{FIGURE2}. The natural phase-space for the outer billiard map is the cilinder $\mathbb{S}^1 \times (0,+\infty)$. We notice that $\Phi$ is continuous and can be continuously extended to $\mathbb{S}^1 \times [0,+\infty)$ by fixing $\mathbb{S} \times \{0\}$. We set $\rho = 1/2$ at $+\infty$ so that the twist interval is $(0,1/2)$. Moreover --we refer e.g. to \cite{BiMi}-- also such a $\Phi$ satisfies the twist condition, with the area of the curvilinear triangle of vertices $x, \Phi(A)$ and $y$ (see Figure \ref{FIGURE2} again) as a generating function. In view all these facts, the marked area spectrum for the outer billiard map is defined as follows. 

\begin{definition}
The marked area spectrum for the outer billiard is the map
$$\mathcal{MA}_o(C): \mathbb{Q} \cap (0, \frac{1}{2}\Big{)} \to \mathbb{R}$$
that associates to any $m/n$ in lowest terms the minimal area of the periodic trajectories having rotation number $m/n$. 
\end{definition}

\noindent We notice that periodic outer billiard minimal trajectories (with winding number $=1$) correspond to convex polygons realizing the minimal (circumscribed) area, the so-called best approximating circumscribed polygons. Analogously to the previous section, we denote by $\mathcal{P}_n^c$ the set of circumscribed convex polygons with at most $n$ vertices and 
\begin{equation}\label{///}
\delta(C, \mathcal{P}_n^c):=inf\lbrace \delta(C,P_n) : P_n\in\mathcal{P}_n^i\rbrace
\end{equation}
where $\delta(C,P_n)$ is the area of the complement of $C$ in $P_n$. Consequently:
\begin{equation} \label{beta-MB}
    \beta\left(\frac{1}{n}\right)=\frac{1}{n}\left(\mathcal{MA}_o(C)\left(\frac{1}{n}\right)-Area(C)\right)=\frac{1}{n}\delta(C,\mathcal{P}_n^c).
\end{equation}

\begin{figure}[H] \label{FIGURE2}
 \centering
    \begin{tikzpicture}
    \draw[color=black] (0,0) ellipse (3cm and 2cm);
    \draw[semithick] (-2.5,1.12) -- (-0.5,3.07);
    \draw[semithick] (-0.5,3.07) -- (2,1.5);
    \draw[semithick, dashed] (2,1.5)  -- (3.5, 0.5);
	\filldraw[black] (-0.5,3.07) circle(1.5pt) node[above=3pt] {$\Phi(A)$};
    \filldraw[black] (2,1.5) circle(1.5pt) node[above=3pt] {$y$};
	\filldraw[black] (-2.5,1.12) circle(1.5pt) node[above=3pt] {$x$};
    \draw[semithick] (-2.5,1.12) -- (-4.45,-0.83);
    \filldraw[black] (-4.45,-0.83) circle(1.5pt) node[above=3pt] {$A$};    
    \end{tikzpicture}
    \caption{The outer billiard map reflection.}
\end{figure}
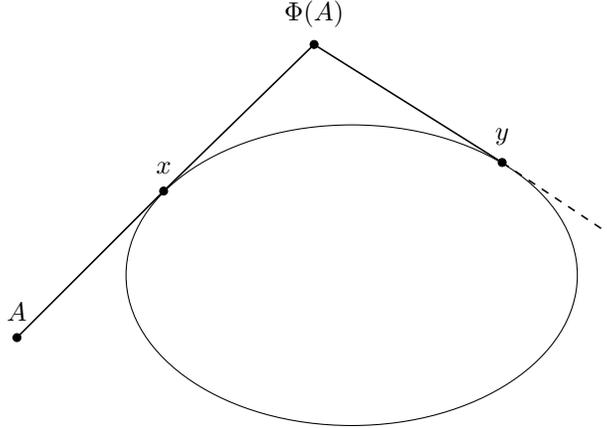   

\subsection{Main result} \label{MAIN}
S. Marvizi and R. Melrose's theory, first stated and proved for Birkhoff billiards \cite{MaMe}[Theorem 3.2], applies both to symplectic and outer billiards (see \cite{ALB}[Section 2.5] and \cite{Tab}[Section 8] respectively). It assures that the corresponding dynamics equals to the time-one flow of a Hamiltonian vector field composed with a smooth map fixing pointwise the boundary of the phase-space at all orders. We refer to \cite{Glu}[Section 2.1] for a detailed proof in the general case of (strongly) billiard-like maps. As an outcome, this result gives the following expansion at $\rho = 0$ of the corresponding minimal average function
$$\beta\left(\rho\right) \sim \beta_1 \rho + \beta_3 \rho^3 + \beta_5 \rho^5 + \beta_7 \rho^7 + \ldots$$
in terms of odd powers of $\rho$. It is well-known --see e.g. \cite{MaMe}[Section 7] again-- that for usual billiards the sequence $\{\beta_{k}\}$ can be interpreted as a spectrum of a differential operator, see also Remark 2.11 in \cite{ALB}. The question is open for symplectic and outer billiards. Coefficients $\beta_1, \ldots, \beta_5$ are known. We refer e.g to \cite{ML} for detailed computations: Theorem 1 for the symplectic case and Theorem 2 for the outer one. Moreover, we suggest \cite{ATnew} for a recent discussion on this topics. The main result of the present paper, stated in the next theorem, is providing coefficients $\beta_7$ both in the symplectic and outer case.

\begin{theorem}\label{TH} Let $C$ be a strictly-convex planar domain with smooth boundary $\partial C$. Suppose that $\partial C$ has everywhere positive curvature. Denote by $k(s)$  the affine curvature  of $\partial C$ with affine parameter $s$. Let $\lambda$ be the affine length of the boundary. 

\begin{enumerate}
\item[$(a)$] The formal Taylor expansion at $\rho = 0$ of Mather's $\beta$-function for the symplectic billiard map has coefficients:
\begin{equation*}
\begin{split}
&\beta_{2k} = 0 \text{ for all } k \\
    &\beta_1=-2Area(C)\\
        &\beta_3=\frac{\lambda^3}{6}\\
         &\beta_5=-\frac{\lambda^4}{5!}\int_0^\lambda k(s)ds\\
        &\beta_7=-\frac{9\lambda^6}{5\cdot7!}\int_0^\lambda k^2(s) ds+\frac{\lambda^5}{15\cdot5!}\left(\int_0^\lambda k(s) ds\right)^2.
    \end{split}
    \end{equation*}
\item[$(b)$] The formal Taylor expansion at $\rho = 0$ of Mather's $\beta$-function for the outer billiard map has coefficients:
\begin{equation*}
\begin{split}
&\beta_{2k} = 0 \text{ for all } k \\
    &\beta_1=0\\
        &\beta_3=\frac{\lambda^3}{24}\\
         &\beta_5=\frac{\lambda^4}{2\cdot5!}\int_0^\lambda k(s)ds\\
        &\beta_7=\frac{421\lambda^6}{5\cdot8!}\int_0^\lambda k^2(s) ds-\frac{\lambda^5}{5\cdot5!}\left(\int_0^\lambda k(s) ds\right)^2.
    \end{split}
    \end{equation*}
\end{enumerate}    
\end{theorem}
\noindent As a straightforward consequence, in Corollary \ref{TAB} we point out that the two coefficients $\beta_5$ and $\beta_7$ allow to recognize an ellipse among all strictly-convex planar domains, both for symplectic and outer billiards.


\section{Preliminaries of affine differential geometry}
Since the areas $\delta(C, \mathcal{P}_n^i)$ and $\delta(C, \mathcal{P}_n^c)$ --defined in (\ref{7}) and (\ref{///}) respectively-- are invariant with respect to affine transformations, we use the affine arc length to parametrize $C$. \\
\indent We recall that the affine arc length and the affine curvature are given respectively by 
$$s(t)=\int_0^t \kappa^{\frac{1}{3}}(\tau) d\tau\qquad 0\leq t\leq l$$
and 
$$\lambda=\int_0^l \kappa^{\frac{1}{3}}(\tau)d\tau$$
where $t$ is the ordinary arc length and $\kappa(t)$ the curvature of $C$. In the following, let $x(s)$ be an affine arc length parametrization of $\partial C$. We denote by $x^{(i)}$ the $i$-th derivative of $x$ and we omit the dependence on $s$. Then we have (see \cite{AG}[Section 3]):
\begin{equation}\label{1}
   \omega(x^{(1)}, x^{(2)})=1, \qquad
 \omega(x^{(1)}, x^{(3)})=0   \end{equation} 

\noindent and 
\begin{equation}\label{3}
k(s) := \omega(x^{(2)},x^{(3)})
\end{equation}
is called the affine curvature of $C$. \\
\indent The next lemma, which is a straightforward consequence of formulae (\ref{1}) and (\ref{3}), will be very useful through the paper.
\begin{lemma}  \label{fatti tec}
The following relations hold:  
\begin{equation} \label{4}
    \omega(x^{(1)}, x^{(4)})=-k, \quad \omega(x^{(2)}, x^{(4)})=k', \quad \omega(x^{(1)}, x^{(5)})=-2k'
\end{equation}
and
\begin{equation} \label{5}
\omega(x^{(3)}, x^{(4)})=k^2,\quad\omega(x^{(2)}, x^{(5)})=k''-k^2\quad \omega(x^{(1)}, x^{(6)})=-3k''+k^2.
\end{equation}
\end{lemma}
\begin{proof} By deriving the second equality in \eqref{1} and taking into account \eqref{3} we have that $\omega(x^{(1)}, x^{(4)})=-k$. 
Similarly, from \eqref{3} we get $ \omega(x^{(2)},x^{(4)})=k'$. Moreover, by deriving the first  identity of \eqref{4} we obtain the third one. In order to obtain the relations in \eqref{5} it is enough to recall that $x^{(3)}=-kx^{(1)}$ and to derive the identities in \eqref{4}.
\end{proof}

\section{Asymptotic expansion for $\delta(C, \mathcal{P}_n^i)$}
We gather in this section all the technical results in order to prove point $(a)$ of Theorem \ref{TH}.

We begin with a refinement of Lemma 1 in \cite{ML}.
\begin{proposition}\label{L1} For $0\leq r\leq s\leq\lambda$, let $F(r,s)$ be the area of the region between the arc $\{ x(t):\   r\le t\le s\}$, and the line segment with end points $x(r)$ and $x(s)$. Then $$F(r,s)=\frac{1}{2}\left( \frac{(s-r)^3}{3!}-\frac{(s-r)^5}{5!}k(r)-\frac{3(s-r)^6}{6!}k'(r)-\frac{(s-r)^7}{7\cdot5!}k''(r)+\frac{(s-r)^7}{7!}k^2(r)+o((s-r)^7)\right)$$uniformly for all $0\leq r\leq s\leq \lambda$ as $(s-r)\to0$.
\end{proposition}
\begin{proof}
    Without loss of generality, we assume $r=0$. The area  of the region of $C$ bounded by the segment $[x(0),x(s)]$  is given by 
    \begin{equation} \label{6}
    F(s):=F(0,s)=\frac{1}{2}\int_0^s\omega(x(t)-x(0),x^{(1)}(t))dt.
    \end{equation}
    We consider the Taylor expansion of the function inside the integral and we have  
     \begin{equation} \label{taylor}
     \omega(x(t)-x(0),x^{(1)}(t))=\sum^m_{l=2}\left(\sum_{\substack{j+h-1=l\\1\leq j< h}} \frac{(h-j)\omega(x^{(j)}(0),x^{(h)}(0))}{j!h!}\right)t^l+o(t^m).
     \end{equation}

\noindent By Lemma \ref{fatti tec} and formula \eqref{taylor} we get 
$$ \omega\left(x(t)-x(0),x'(t)\right) = \frac{t^2}{2} -\frac{k(0)}{4!}t^4 -\frac{3k'(0)}{5!}t^5 +\left( \frac{k^2(0)}{6!}-\frac{k''(0)}{5!}\right) t^6 +o(t^6)$$
and the statement follows after integration on $0\le t\le s$. 
\end{proof}

The remaining of this section is devoted to characterize the affine lengths of the $n$ arcs in which $\partial C$ is divided by the vertices of a best approximating polygon $P_n$ inscribed in $C$. 
\begin{proposition}\label{L3} For $n\ge 3$ let $P_n\in\mathcal{P}_n^i$ be a best approximating polygon inscribed in $C$. Let $x(s_{n,i})$, $i = 1, \ldots, n$, be the ordered vertices of $P_n$ and 
$$\lambda_{n,i}:=s_{n,i}-s_{n,i-1}.$$ 
Then \begin{equation}\label{eq1}
    \lambda_{n,i+1}=\lambda_{n,i}+\frac{1}{30}k'(s_{n,i})\lambda_{n,i}^4+\epsilon_1(s_{n,i},\lambda_{n,i})
\end{equation} where  $\lim_{n\to\infty}\frac{\epsilon_1(s_{n,i},\lambda_{n,i})}{\lambda_{n,i}^5}=0$ uniformly in $i$.   
\end{proposition}
\begin{proof}
Since $P_n$ is a best approximating polygon, its vertices satisfy: 
\begin{equation}\label{8}
    \omega\left(x(s_{n,i-1})-x(s_{n,i+1}),x^{(1)}(s_{n,i})\right)=0 \quad \text{for }i=1,...,n
\end{equation}
To study \eqref{8}, we assume $s_{n,i}=0$ and consider the equivalent equation:
$$ x(s)-x(t)+hx^{(1)}(0)=0  \text{ for some } h\in \mathbb{R}.
$$
Excluding the trivial solution $t \equiv s$, we obtain (up to rename $h$):
\begin{equation}\label{eq}
    \frac{x(s)-x(t)}{s-t}+hx^{(1)}(0)=0.
\end{equation} Let us define $$G(s,t):=\frac{x(s)-x(t)}{s-t}$$ and extend it, smoothly, to $s=t$ by setting $G(s,s)=x^{(1)}(s)$. Thus we have 
$$G(s,t)=\frac{x(s)-x(t)}{s-t}=\sum_{k=1}^m \frac{x^{(k)}(0)}{k!}\left(\frac{s^k-t^k}{s-t}\right) +o((s-t)^{m-1})$$
and $s=t=0$, $h=-1$ is a solution for \eqref{eq}. In order to apply the Implicit Function Theorem to solve \eqref{eq} for $h$ and $t$ in terms of $s$, we compute the Jacobian matrix of $G(s,t)+hx^{(1)}(0)$ in $s=t=0$, $h=-1$: $$J_{h,t}\big(G(s,t)+hx^{(1)}(0)\big)_{|s=t=0,h=-1}=\left(x^{(1)}(0),\frac{x^{(2)}(0)}{2}\right).$$ 
Since $detJ_{h,t}=\frac{1}{2}$, it is possible to solve (\ref{eq}) and find $h(s)$ and $t(s)$. Deriving we get
\begin{equation} \label{GI}
\frac{d}{ds}\left(G(s, t(s))+h(s)x^{(1)}(0)\right) = \partial_sG(s,t(s))+\partial_tG(s,t(s))t'(s)+h'(s)x^{(1)}(0)=0.
\end{equation}
Since for $s=0$, $t(0)=0$, we have
$$\frac{x^{(2)}(0)}{2}+\frac{x^{(2)}(0)}{2}t'(0)+h'(0)x^{(1)}(0)=0,$$
and therefore we get $t'(0)=-1$. \\
In order to obtain the higher order derivatives of $t$ in $0$ it is enough to argue in the same way that is deriving \eqref{GI} again and taking into account Lemma \ref{fatti tec}. In such a way we get $$t(s)=-s+\frac{1}{30}k'(0)s^4+o(s^5).$$
\noindent The previous expansion written for general $s_{n,i} \ne 0$ gives immediately formula (\ref{eq1}).

\comment{For the second derivative we get 
\begin{equation} \label{eq2}
 \partial_s^2G+2\partial_t\partial_sGt'(s)+\partial_t^2G(t'(s))^2+\partial_tGt''(s)+h''(s)x^{(1)}(s)=0  
\end{equation}
where$$\partial_s^2G(0,0)=\frac{x^{(3)}(0)}{6}\cdot2=\partial_t^2G(0,0),\qquad \partial_{s,t}^2G(0,0)=\frac{x^{(3)}(0)}{6},$$
and so $$x^{(3)}(0)\left(\frac{1}{3}+\frac{1}{3}+\frac{1}{3}\right)+\frac{x^{(2)}(0)}{2}t''(0)+h''(0)x^{(1)}(0)=0.$$Then from $\omega(x^{(3)}(0),x^{(1)}(0))=0$, we get $t''(0)=0$.\\

In the next term we get $$\sum_{i=0}^3\partial_s^i\partial_t^{3-i}G(s,t)\binom{3}{i}(t'(s))^{3-i}+\left(\dots\right)t''(s)+\partial_tGt'''(s)+h'''(s)x^{(1)}(s)=0$$and $$\partial_s^i\partial_t^{3-i}G(0,0)=\frac{x^{(4)}(0)}{4!}i!(3-i)!.$$Then for $s=0$ we get $$\sum_{i=0}^3\frac{x^{(4)}(0)}{4!}3!(-1)^{3-i}+\frac{x^{(2)}(0)}{2}t'''(0)+h'''(0)x^{(1)}(0)=0$$and so $t'''(0)=0$.\\

Considering the fourth derivative the only term which remains is 
$$\sum_{i=0}^4\frac{x^{(5)}(0)}{5!}4!(-1)^{4-i}+\frac{x^{(2)}(0)}{2}t^{(4)}(0)+h^{(4)}(0)x'(0)=0,$$ $$\frac{x^{(5)}(0)}{5}+\frac{x^{(2)}(0)}{2}t^{(4)}(0)+h^{(4)}(0)x'(0)=0,$$then $t^{(4)}(0)=-\frac{2}{5}\omega(x'(0),x^{(5)}(0))=\frac{4}{5}k'(0).$\\

For the computation of $t^{(5)}(0)$ we need  to derive three times \eqref{eq2} and to keep track only of the terms that involve $t'(0)$ and $t^{(4)}(0)$. We easily get that $t^{(5)}(0)=0$.} 
\end{proof}
\noindent The next proposition is a refinement of Lemma 2 in \cite{ML}.
\begin{proposition}\label{L2} Under the same assumption of the previous proposition, it holds 
\begin{equation} \label{finalmente}
\lambda_{n,i}=\frac{\lambda}{n}-\frac{\lambda^2}{30 n^3} \int_0^\lambda k(s)ds+\frac{\lambda^3}{30 n^3}k\left(\frac{i\lambda}{n}\right)+o\left(\frac{1}{n^3}\right)
\end{equation}
uniformly in $i$ as $n\to\infty$.
\end{proposition}
\begin{proof} By a standard comparison argument of difference equations, we first prove that 
\begin{equation}\label{luglio}
        \lambda_{n,i}=\frac{\lambda}{n}+O\left(\frac{1}{n^3}\right) \qquad \text{uniformly in }i\text{ as }n\to\infty
    \end{equation}
\noindent Let $D$ be a constant such that $D> \frac{1}{30} \max |k'|$ and consider the two difference equations
\begin{equation}\label{luglio1}
    v_{n,i+1}=v_{n,i}\left(1-Dv_{n,i}^3\right),\quad V_{n,i+1}=V_{n,i}\left(1+DV_{n,i}^3\right).
\end{equation}
In view of Lemma 2 in \cite{ML}, $\lambda_{n,1}=\frac{c_n}{n}$ for some uniform constant.  By taking the first terms of  \eqref{luglio1} both equal to $\frac{c_n}{n}$, it clearly holds 
$$v_{n,i}\le \lambda_{n,i}\le V_{n,i}, \qquad i=1,...,n.$$
Moreover, the sequence $V_{n,i}$ is increasing and $V_{n,n}$ is uniformly bounded. In fact, let $k\geq 1$ be the largest integer such that $V_{n,k}\leq \frac{2c_n}{n}$. If $k<n$, then we should have $V_{n,k+1}>\frac{2c_n}{n}$ and therefore
    $$\frac{2c_n}{n}<\frac{c_n}{n}\left(1+D\left(\frac{2c_n}{n}\right)^3\right)^k<\frac{c_n}{n}\left(1+D\left(\frac{2c_n}{n}\right)^3\right)^n \Rightarrow 
    2<\left(1+D\left(\frac{2c_n}{n}\right)^3\right)^n$$
Since the last term tends to $1$, this contradicts the fact that $k<n$. \\
With an analogous argument, it follows that the sequence $v_{n,i}$ is decreasing and $v_{n,n}$ is uniformly bounded. \\
\comment{Starting from the difference equation \ref{eq1}, we initially prove that its solutions $\lambda_1,\dots,\lambda_n$ with initial data $\lambda_0=\frac{c}{n}$ (by Lemma 2 of \cite{ML}) are uniformly bounded. In order to do that we consider the difference equation \begin{equation}
    v_{i+1}=v_i\left(1+Dv_i^3\right)
\end{equation}
    with $D>0$ constant uniform in $i$ and big enough to have the solution $v_i$ of it bigger than the ones of the equation \ref{eq1}. The solution $\lbrace v_i\rbrace$ is increasing and considering $v_0=\frac{c}{n}$ we know that there exists an integer $k\geq0$ such that for $i\leq k$, $v_k\leq\frac{2c}{n}$. Assume such $k<n$, then we should get $$v_{k+1}>\frac{2c}{n}$$
    $$\frac{2c}{n}<\frac{c}{n}\left(1+D\left(\frac{2c}{n}\right)^3\right)^k<\frac{c}{n}\left(1+D\left(\frac{2c}{n}\right)^3\right)^n$$
    $$2<\left(1+D\left(\frac{2c}{n}\right)^3\right)^n$$
    but the last term, for $n$ big enough, tends to $1$ and this contradicts the hypothesis $k<n$. Then $v_i\leq\frac{2c}{n}$ for $i=1,\dots,n$.\\
    To obtain an estimate from below we study the difference equation \begin{equation}
    v_{i+1}=v_i\left(1-Dv_i^3\right)
\end{equation}
whose solutions are smaller then the ones of \ref{eq1}. The estimate that we obtain with initial data $v_0=\frac{c}{n}$, is $v_i\geq\frac{c}{2n}$ for $i=1,\dots,n$.}
Therefore, up to rename the constant $D$, we have
$$\lambda_{n,1}\left(1-\frac{D}{n^3}\right)^{i-1}\leq\lambda_{n,i}\leq \lambda_{n,1} \left(1+\frac{D}{n^3}\right)^{i-1}.$$
Summing for $i=1,\dots,n$, we get 
$$\lambda_{n,1}\frac{1-\left(1-\frac{D}{n^3}\right)^n}{\frac{D}{n^3}}\leq \lambda \leq\lambda_{n,1}\frac{1-\left(1+\frac{D}{n^3}\right)^n}{-\frac{D}{n^3}} \Rightarrow
\lambda_{n,1}\frac{n}{1+\frac{D}{n^2}}\leq \lambda \leq \lambda_{n,1}\frac{n}{1-\frac{D}{n^2}}$$
\noindent so that $\lambda_{n,1}=\frac{\lambda}{n}+O\left(\frac{1}{n^3}\right)$, which corresponds to formula (\ref{luglio}). Such a formula immediately gives $s_{n,i} = \frac{i\lambda}{n} + O\left(\frac{1}{n^2}\right)$ and therefore --up to renaming the function $\epsilon_1$-- expansion (\ref{eq1}) of Proposition \ref{L3} can be equivalently written as
\begin{equation}\label{diffeq1}
    u_{n,i+1} = u_{n,i} + \frac{1}{30}k'\left(\frac{i\lambda}{n}\right)\frac{u_{n,i}^4}{n^3}+\epsilon_1\left(\frac{i\lambda}{n},\frac{u_{n,i}}{n}\right)
\end{equation}
where $u_{n,i} = n \, \lambda_{n,i} \,$.
Let $u_n$ be the solution of the Cauchy problem (corresponding to (\ref{diffeq1})):
\begin{equation}\label{ode1}
    \begin{cases}
        u'_n(t)=\frac{1}{30 \lambda}k'(t)\frac{u_n^4(t)}{n^2} \\
        u_n\left(\frac{\lambda}{n}\right)=u_{n,1}
    \end{cases}
\end{equation}
and set $u_n(i):=u_n \left(\frac{i\lambda}{n}\right)$. \\
\indent The second part of the proof is devoted to establish the next estimate:
\begin{equation} \label{stima}
u_n(i)-u_{n,i}=o\left( \frac{1}{n^2}\right)
\end{equation}
uniformly in $i = 1, \ldots, n$ as $n \to \infty$. \\
By integrating (\ref{ode1}) between $\frac{i\lambda}{n}$ and $\frac{(i+1)\lambda}{n}$ via separation of variables, we get
$$-\frac{1}{3}\left(\frac{1}{u_n(i+1)^3}-\frac{1}{u_n(i)^3}\right)= \frac{1}{30 \lambda n^2}\left[k\left(\frac{(i+1)\lambda}{n}\right)-k\left(\frac{i\lambda}{n}\right)\right],$$ 
that is
$$\left(u_n(i+1)\right)^{-3} = \frac{1}{10 \lambda n^2} \left[ k\left(\frac{i\lambda}{n}\right) - k\left(\frac{(i+1) \lambda }{n}\right) \right] + u_n(i)^{-3},$$ 
implying
\begin{equation*}
\begin{split}
  u_n(i+1) &= u_n(i) \left({1+\frac{1}{10 \lambda n^2} \left[ k\left(\frac{i\lambda}{n}\right)-k\left(\frac{(i+1)\lambda }{n}\right)\right] u_n(i)^3} \right)^{-1/3} \\
  &= u_n(i) \left(1+ \frac{1}{30 \lambda n^2}\left[k\left(\frac{(i+1)\lambda}{n}\right)-k\left(\frac{i\lambda}{n}\right)\right]u_n(i)^3\right) + O\left(\frac{1}{n^4}\right) \\
  &= u_n(i) + \frac{1}{30n^3}k'\left(\frac{i\lambda}{n}\right) u_n(i)^4 + o\left(\frac{1}{n^3}\right).
\end{split}
\end{equation*}
Taking into account the previous formula and the difference equation (\ref{diffeq1}), we immediately have 
$$u_{n,i+1}-u_n(i+1)=u_{n,i}-u_n(i) + \frac{1}{30n^3}k'\left(\frac{i\lambda}{n}\right)(u_{n,i}^4 - u_n(i)^4)+o\left(\frac{1}{n^3}\right).$$
Equivalently, $z_{n,i} := u_{n,i}-u_n(i)$ solves 
\begin{equation}\label{eqz}
    z_{n,i+1}=z_{n,i}+\frac{C_{n,i}}{30n^3}k'\left(\frac{i\lambda}{n}\right)z_i +o\left(\frac{1}{n^3}\right)
\end{equation}
where $C_{n,i}>0$ are constants uniformly bounded in $i$ for $n \to \infty$. As in the first part of the proof, let $E$ be a constant such that $E > \frac{C_{n,i}}{30} \max |k'|$ (for every $i$ and $n$) and consider the difference equation
$$Z_{n,i+1}=Z_{n,i}\left(1+\frac{E}{n^3}\right)+\frac{c_n}{n^3},$$
with $c_n = o(1)$ and $Z_{n,1}=0$. Comparing the $z_{n,i+1}$'s in (\ref{eqz}) with the terms $Z_{n,i+1}$ of the previous difference equation, for every $i = 1, \ldots, n$, we get
$$z_{n,i+1} \le Z_{n,i+1} = \frac{c_n}{E}\left(\left(1+\frac{E}{n^3}\right)^i-1\right) \leq o(1)\left(\frac{1}{1-\frac{iE}{n^3}}-1\right) = o(1)\frac{i}{n^3} = o\left(\frac{1}{n^2}\right),$$
which corresponds to (\ref{stima}). \\
\indent We finally explicit the term of order $3$ in formula (\ref{finalmente}). By integrating (\ref{ode1}) between $\frac{\lambda}{n}$ and $\frac{(i+1)\lambda}{n}$, we have 
\begin{equation*}
\begin{split}
  u_n(i+1) &= u_n(1) \left({1+\frac{1}{10 \lambda n^2} \left[ k\left(\frac{\lambda}{n}\right)-k\left(\frac{(i+1)\lambda }{n}\right)\right] u_n(1)^3} \right)^{-1/3} \\
  &= u_n(1) \left(1+ \frac{1}{30 \lambda n^2}\left[k\left(\frac{(i+1)\lambda}{n}\right)-k\left(\frac{\lambda}{n}\right)\right]u_n(1)^3\right) + O\left(\frac{1}{n^4}\right).
\end{split}
\end{equation*}
Moreover, by formula (\ref{stima}), $u_{n,i+1} = u_{n}(i+1) + o \left( \frac{1}{n^2}\right)$ or equivalently (since $\lambda_{n,i} = \frac{u_{n,i}}{n}$)
$$\lambda_{n,i+1} = \frac{u_{n}(i+1)}{n} + o \left( \frac{1}{n^3}\right).$$
Plugging the previous expression of $u_n(i+1)$ (in terms of $u_n(1) = n \, \lambda_{n,1}$) into formula above, we obtain
$$\lambda_{n,i+1} = \lambda_{n,1} \left[ 1 + \frac{1}{30\lambda n^2} \left( k\left(\frac{(i+1)\lambda}{n}\right)-k\left(\frac{\lambda }{n}\right) \right) n^3 \lambda_{n,1}^3 \right] + o \left( \frac{1}{n^3} \right).$$
In more detail, since $\lambda_{n,1} = \frac{\lambda}{n} + \frac{e_n}{n^3} + o \left( \frac{1}{n^3} \right)$ for some uniform constant (see formula (\ref{luglio})), we have
\begin{equation}\label{lambda n 1}
\lambda_{n,i+1} = \frac{\lambda}{n} + \frac{\lambda^3}{30 n^3} \left( k\left(\frac{(i+1)\lambda}{n}\right)-k\left(\frac{\lambda }{n}\right) \right) + \frac{e_n}{n^3} + o \left( \frac{1}{n^3} \right).
\end{equation}
Summing for $i = 0, \ldots, n-1$, we conclude that
$$\lambda = \sum_{i = 0}^{n-1} \lambda_{n,i+1} = \lambda + \frac{\lambda^3}{30n^3} \sum_{i = 0}^{n-1} \left( k\left(\frac{(i+1)\lambda}{n}\right)-k\left(\frac{\lambda }{n}\right) \right) + \frac{e_n}{n^2} + o \left( \frac{1}{n^2} \right),$$
that is
$$e_n=-\frac{\lambda^3}{30n}\sum_{i=0}^{n-1}
\left( k\left(\frac{(i+1)\lambda}{n}\right)-k\left(\frac{\lambda }{n}\right) \right)
+o(1)$$
and the limit for $n\to\infty$ leads to $$e_n=-\frac{\lambda^2}{30}\int_0^\lambda k(s)-k\left(0\right)ds.$$ 
Plugging the expression for $e_n$ into formula (\ref{lambda n 1}), we finally obtain (for $n\to\infty$):
\begin{equation*}
\lambda_{n,i+1} = \frac{\lambda}{n} - \frac{\lambda^2}{30 n^3} \int_0^\lambda k(s) ds +  \frac{\lambda^3}{30 n^3}  k\left(\frac{(i+1)\lambda}{n}\right)   + o \left( \frac{1}{n^3} \right),
\end{equation*}
which equals to (\ref{finalmente}). 
\end{proof}

\section{Asymptotic expansion for $\delta(C, \mathcal{P}_n^c)$}
In this section, analogously to the previous one, we collect the technical results in order to prove point $(b)$ of Theorem \ref{TH}. \\
 \indent The next result is the analogous of Proposition \ref{L1} for circumscribed polygons.
\begin{proposition}
    For $0\leq r\leq s\leq\lambda$, let $H(r,s)$ be the area of the region between the tangents in $x(r)$ and $x(s)$ and the  arc $\{ x(t) : r\le t\le s\}$. Then $$H(r,s)=\frac{(s-r)^3}{24}+\frac{(s-r)^5}{2\cdot5!}k(r)+\frac{3(s-r)^6}{6!}k'(r)+\frac{(s-r)^7}{7\cdot5!}k''(r)+\frac{17 (s-r)^7}{8!}k^2(r)+o((s-r)^7))$$uniformly for all $0\leq r\leq s\leq \lambda$ as $(s-r)\to0$.
\end{proposition}
\begin{proof}
    Without loss of generality, we assume $r=0$. Consider the vertex $P$ of the polygon $P_n\in\mathcal{P}_n^c$ whose edges are tangent to $C$ in $x(0)$ and $x(s)$: $$P=x(s)+t_1x'(s)=x(0)+t_2x'(0) \Rightarrow \omega(x'(0), x(0)-x(s))=t_1\omega(x'(0),x'(s))$$so we get\begin{equation}
        P=x(s)+x'(s)\frac{\omega(x'(0), x(0)-x(s))}{\omega(x'(0),x'(s))}.
    \end{equation}
    The area of the triangle of vertices $x(0)$, $x(s)$ and $P$ is 
    \begin{equation*}
    \begin{split}-\frac{1}{2}\omega(P-x(s), x(0)-x(s))
    &=\frac{1}{2}\frac{\omega(x'(s), x(s)-x(0))\omega(x'(0), x(0)-x(s))}{\omega(x'(0),x'(s))}.
    \end{split}        
    \end{equation*}
    Taking $M,N,L,J\in\mathbb{N}$ large enough in order to approximate the above quantity at order $7$, we obtain
    $$\frac{1}{2}\omega\left(\sum_{h_1=1}^M \frac{x^{(h_1)}(0)s^{h_1-1}}{(h_1-1)!}+o(s^M),\sum_{h_2=1}^N \frac{x^{(h_2)}(0)s^{h_2}}{h_2!}+o(s^N)\right)\left(-\sum_{h_3=2}^L\frac{\omega(x'(0),x^{(h_3)}(0))s^{h_3-1}}{h_3!}+o(s^L)\right)\cdot$$$$\cdot\left(\sum_{h_4=2}^J\frac{\omega(x'(0),x^{(h_4)}(0))s^{h_4-2}}{(h_4-1)!}+o(s^J)\right)^{-1}=$$
    $$=\frac{1}{2}\left(\sum^M_{m=1}\sum_{j+h-1=m}\frac{\omega(x^{(h)}(0),x^{(j)}(0))s^{m}}{j!(h-1)!}+o(s^M)\right)\left(-\sum_{h_3=2}^L\frac{\omega(x'(0),x^{(h_3)}(0))s^{h_3-1}}{h_3!}+o(s^L)\right)\cdot$$$$\cdot\left(\sum_{h_4=2}^J\frac{\omega(x'(0),x^{(h_4)}(0))s^{h_4-2}}{(h_4-1)!}+o(s^J)\right)^{-1}=$$
    $$=\frac{1}{2}\left(\frac{1}{2}s^2-\frac{k(0)}{4!}s^4-\frac{3k'(0)}{5!}s^5-\frac{k''(0)}{5!}s^6+\frac{k^2(0)}{6!}s^6+o(s^6)\right)\biggl(\frac{1}{2}s-\frac{k(0)}{4!}s^3-\frac{2k'(0)}{5!}s^4+$$ $$-\frac{3k''(0)}{6!}s^5+\frac{k^2(0)}{6!}s^5+o(s^5)\biggr)\left(1-\frac{k(0)}{3!}s^2-\frac{2k'(0)}{4!}s^3-\frac{3k''(0)}{5!}s^4+\frac{k^2(0)}{5!}s^4+o(s^4)\right)^{-1}=$$
    $$=\frac{1}{2}\left(\frac{1}{2}s^2-\frac{k(0)}{4!}s^4-\frac{3k'(0)}{5!}s^5-\frac{k''(0)}{5!}s^6+\frac{k^2(0)}{6!}s^6+o(s^6)\right)\biggl(\frac{1}{2}s-\frac{k(0)}{4!}s^3-\frac{2k'(0)}{5!}s^4+$$ $$-\frac{3k''(0)}{6!}s^5+\frac{k^2(0)}{6!}s^5+o(s^5) \biggr) \left(1+\frac{k(0)}{3!}s^2+\frac{2k'(0)}{4!}s^3+\frac{3k''(0)}{5!}s^4-\frac{k^2(0)}{5!}s^4+\frac{k^2(0)}{36}s^4+o(s^4)\right)=$$
    
    $$=\frac{1}{2}\left(\frac{1}{4}s^3+\frac{3k^2(0)}{4\cdot6!}s^7+o(s^7)\right).$$
By difference of areas and taking into account the result of Proposition \ref{L1} for $F(0,s)$, we get
    $$H(0,s)=
    \frac{1}{24}s^3+\frac{k(0)}{2\cdot5!}s^5+\frac{k'(0)}{4\cdot5!}s^6+\frac{k''(0)}{14\cdot5!}s^7+\frac{17k^2(0)}{8!}s^7+o(s^7).$$
\end{proof}

\indent The next two results correspond to Propositions \ref{L3} and \ref{L2} in the case of circumscribed best approximating polygons. 
\begin{proposition}\label{L4} For $n\ge 3$ let $P_n\in\mathcal{P}_n^c$ be a best approximating polygon circumscribed to $C$. Let $x(s_{n,i})$, $i = 1, \ldots, n$, be the ordered tangency points of the edges of $P_n$ to $C$ and 
$$\lambda_{n,i}:=s_{n,i}-s_{n,i-1}.$$ 
Then \begin{equation}\label{eq2}
   \lambda_{n,i+1}=\lambda_{n,i}-\frac{8}{5!}k'(s_{n,i})\lambda_{n,i}^4+\epsilon_2(s_{n,i},\lambda_{n,i})
\end{equation} where  $\lim_{n\to\infty}\frac{\epsilon_2(s_{n,i},\lambda_{n,i})}{\lambda_{n,i}^5}=0$ uniformly in $i$.   
\end{proposition}
\begin{proof}
If $x(s_{n,1})$, $x(s_{n,2})$ and $x(s_{n,3})$, $0\leq s_{n,i}\leq\lambda$, are three successive tangent points of $P_n$ to $C$, the corresponding vertices of $P_n$ are $$P=x(s_{n,2})+x'(s_{n,2})\frac{\omega(x'(s_{n,1}),x(s_{n,1})-x(s_{n,2}))}{\omega(x'(s_{n,1}),x'(s_{n,2}))}, \qquad Q=x(s_{n,3})+x'(s_{n,3})\frac{\omega(x'(s_{n,3}),x(s_{n,3})-x(s_{n,2}))}{\omega(x'(s_{n,3}),x'(s_{n,2}))}.$$
Since $P_n$ is a best approximating circumscribed polygon, we have
$$\frac{\omega(x'(s_{n,1}),x(s_{n,1})-x(s_{n,2}))}{\omega(x'(s_{n,1}),x'(s_{n,2}))}=-\frac{\omega(x'(s_{n,3}),x(s_{n,3})-x(s_{n,2}))}{\omega(x'(s_{n,3}),x'(s_{n,2}))}.$$ 
Setting $s_{n,2}=0$, we define the function $$G(s)=\frac{\omega(x'(s),x(s)-x(0))}{\omega(x'(s),x'(0))}.$$
The proof can be concluded as the one of Proposition \ref{L3}, by solving $G(s)=-G(t)$ and applying the Implicit Function Theorem.
\end{proof}
\noindent Finally, an argument analog to the proof of Proposition \ref{L2} gives the next precise characterization of the $\lambda_{n,i}$'s.
\begin{proposition}
     Under the same assumption of the previous proposition, it holds $$\lambda_{n,i}=\frac{\lambda}{n}-\frac{8}{5!}\frac{\lambda^2}{n^3}\int_0^\lambda k(s)ds+\frac{8}{5!}\frac{\lambda^3}{n^3}k\left(\frac{i\lambda}{n}\right)+o\left(\frac{1}{n^3}\right)$$uniformly in $i$ as $n\to\infty$.
\end{proposition}

\section{Proof of Theorem \ref{TH}} 
We finally state and prove the higher order terms of $\delta(C, \mathcal{P}_n^i)$ and $\delta(C, \mathcal{P}_n^c)$ defined in Sections \ref{SB} and \ref{OB} respectively. This theorem is a refinement of Theorem 1 and Theorem 2 in \cite{ML}. In view of equalities (\ref{beta-MA}) and (\ref{beta-MB}), the proof of Theorem \ref{TH} is a straightforward application of this result.  
\begin{theorem}\label{THM1} Let $C$ be a strictly-convex planar domain with smooth boundary $\partial C$.  Suppose that $\partial C$ has everywhere positive curvature. Denote by $k(s)$  the affine curvature  of $\partial C$ with affine parameter $s$. Let $\lambda$ be the affine length of the boundary. 
\begin{enumerate}
\item[$(a)$] The formal expansion of $\delta(C, \mathcal{P}_n^i)$ at $n\to\infty$ is given by \begin{equation*}
    \delta(C, \mathcal{P}_n^i)= a_2\frac{1}{n^2}+a_4\frac{1}{n^4}\int_0^\lambda k(s)ds+\frac{1}{n^6}\left[a_6\int_0^\lambda k^2(s) ds+b_6\left(\int_0^\lambda k(s) ds\right)^2\right]+o\left(\frac{1}{n^6}\right)
\end{equation*} with coefficients
$$a_2=\frac{1}{12}\lambda^3,\qquad a_4=-\frac{1}{2\cdot5!}\lambda^4, \qquad a_6=-\frac{9}{10\cdot7!}\lambda^6, \qquad b_6= \frac{1}{30\cdot5!}\lambda^5.$$
\item[$(b)$]  The formal expansion of $\delta(C, \mathcal{P}_n^c)$ at $n\to\infty$ is given by
\begin{equation*}
    \delta(C, \mathcal{P}_n^c)= a_2\frac{1}{n^2}+a_4\frac{1}{n^4}\int_0^\lambda k(s)ds+\frac{1}{n^6}\left[a_6\int_0^\lambda k^2(s) ds+b_6\left(\int_0^\lambda k(s) ds\right)^2\right]+o\left(\frac{1}{n^6}\right)
\end{equation*}
with coefficients
\begin{equation*}
a_2=\frac{1}{24}\lambda^3,\qquad a_4=\frac{1}{2\cdot5!}\lambda^4, \qquad a_6=\frac{421}{5\cdot8!}\lambda^6, \qquad b_6=- \frac{1}{5\cdot5!}\lambda^5.
\end{equation*}
\end{enumerate}
\end{theorem}

\begin{proof}
Theorem 1 in \cite{ML} establishes that
$$
    \delta(C, \mathcal{P}_n^i) - \frac{1}{12}\frac{\lambda^3}{n^2} + \frac{1}{2 \cdot 5!}\frac{\lambda^4}{n^4}\int_0^\lambda k(s)ds = o\left(\frac{1}{n^4}\right)
$$
as $n \to \infty$. We determine
\begin{equation} \label{differenza}
\lim_{n \to \infty} 
n^6 \left( \delta(C, \mathcal{P}_n^i) - \frac{1}{12}\frac{\lambda^3}{n^2} + \frac{1}{2 \cdot 5!}\frac{\lambda^4}{n^4}\int_0^\lambda k(s)ds \right).
\end{equation}
By means of Proposition \ref{L1}, we get  
$$\delta(C, \mathcal{P}_n^i)=\frac{1}{2}\sum_{i=1}^n\left[ \frac{\lambda_{n,i}^3}{3!}-\frac{\lambda_{n,i}^5}{5!}k(s_{n,i})-\frac{3\lambda_{n,i}^6}{6!}k'(s_{n,i})-\frac{\lambda_{n,i}^7}{7\cdot5!}k''(s_{n,i})+\frac{\lambda_{n,i}^7}{7!}k^2(s_{n,i})+o(\lambda_{n,i}^7)\right]$$ 
$$=\frac{1}{2}\sum_{i=1}^n\left[\frac{1}{3!}\left(\frac{\lambda}{n}+\left(\lambda_{n,i}-\frac{\lambda}{n}\right)\right)^3-\frac{\lambda_{n,i}^5}{5!}k(s_{n,i}) -\frac{3}{6!}\frac{\lambda^6}{n^6}k'(s_{n,i})-\frac{1}{7\cdot5!}\frac{\lambda^7}{n^7}k''(s_{n,i})+\frac{1}{7!}\frac{\lambda^7}{n^7}k^2(s_{n,i})+o\left(\frac{1}{n^7}\right)\right].$$ 
We compute (\ref{differenza}) by dividing the limit in four terms.\\
\indent First, applying Proposition \ref{L2}, we consider 
$$\frac{1}{2}\frac{1}{3!}\sum_{i=1}^n\left[\frac{\lambda}{n}-\left(\lambda_{n,i}-\frac{\lambda}{n}\right)\right]^3-\frac{1}{12}\frac{\lambda^3}{n^2}=\frac{1}{12}\sum_{i=1}^n\left[\frac{\lambda^3}{n^3}-3\frac{\lambda^2}{n^2}\left(\lambda_{n,i}-\frac{\lambda}{n}\right)+3\frac{\lambda}{n}\left(\lambda_{n,i}-\frac{\lambda}{n}\right)^2-\left(\lambda_{n,i}-\frac{\lambda}{n}\right)^3\right]-\frac{1}{12}\frac{\lambda^3}{n^2}$$
$$=\frac{1}{4}\sum_{i=1}^n\left[\frac{1}{900}\frac{\lambda^5}{n^7}\left(\int_0^\lambda k(s)ds\right)^2+\frac{1}{900}\frac{\lambda^6}{n^7}k^2\left(\frac{i\lambda}{n}\right)-\frac{2}{900}\frac{\lambda^6}{n^7} k\left(\frac{i\lambda}{n}\right) \int_0^\lambda k(s)ds +o\left(\frac{1}{n^7}\right)\right].$$
And the corresponding limit gives: 
        $$\lim_{n\to\infty}\frac{n^6}{4}\sum_{i=1}^n\left[\frac{1}{900}\frac{\lambda^5}{n^7}\left(\int_0^\lambda k(s)ds\right)^2+\frac{1}{900}\frac{\lambda^6}{n^7}k^2\left(\frac{i\lambda}{n}\right)-\frac{2}{900}\frac{\lambda^6}{n^7} k\left(\frac{i\lambda}{n}\right) \int_0^\lambda k(s)ds +o\left(\frac{1}{n^7}\right)\right]$$
        \begin{equation} \label{primo contributo}
        =\frac{1}{4\cdot900}\lambda^6\int_0^\lambda k^2(s)ds-\frac{1}{4\cdot900}\lambda^5\left(\int_0^\lambda k(s)ds\right)^2.
        \end{equation}
\indent Second, we compute the limit: 
$$\lim_{n\to\infty}\frac{n^6}{2}\sum_{i=1}^n\left[-\frac{\lambda_{n,i}^5}{5!}k(s_{n,i})\right]+\frac{1}{2\cdot5!} \lambda^4 n^2\int_0^\lambda k(s)ds$$
        $$=\lim_{n\to\infty} \frac{1}{2\cdot 5!} \lambda^4 n^2\left[\sum_{i=1}^n\left(-k(s_{n,i})\lambda_{n,i}+\int_{s_{n,i}}^{s_{n,i+1}} k(s)ds\right)\right]+\frac{n^6}{2\cdot5!}\sum_{i=1}^n-k(s_{n,i})\lambda_{n,i}\left(\lambda_{n,i}^4-\frac{\lambda^4}{n^4}\right).$$
        The limit of the first summand is zero, in fact:
        $$\lim_{n\to\infty} \frac{1}{2\cdot 5!}{\lambda^4 n^2}\left[\sum_{i=1}^n\left(-k(s_{n,i})\lambda_{n,i}+\int_{s_{n,i}}^{s_{n,i+1}} k(s)ds\right)\right]=\lim_{n\to\infty}\frac{1}{2\cdot 5!}{\lambda^4 n^2}\left[\sum_{i=1}^n\int_{s_{n,i}}^{s_{n,i+1}} \left( k(s)-k(s_{n,i}) \right)ds\right]$$
        $$=\lim_{n\to\infty}\frac{1}{2\cdot 5!}{\lambda^4 n^2} \left[ \sum_{i=1}^n\left(\int_{s_{n,i}}^{s_{n,i+1}} (k'(s_{n,i})(s-s_{n,i})+k''(s_{n,i})\frac{(s-s_{n,i})^2}{2}+o((s-s_{n,i})^2)\right) ds \right]$$ 
        $$=\lim_{n\to\infty}\frac{1}{2\cdot 5!}\lambda^4n^2\left[ \sum^n_{i=1} k'(s_{n,i} )\frac{\lambda_{n,i}^2}{2} +k''(s_{n,i})\frac{\lambda_{n,i}^3}{3!}\right] $$
        where, in the last equality, we dropped $o((s-s_{n,i})^2)$ because after integration and summation it gives a negligible term. Now --again by Proposition \ref{L2}-- the above limit equals to        
        $$=\lim_{n\to\infty} \frac{1}{2\cdot 5!}{\lambda^4 n^2} \left[ \sum_{i=1}^n k'(s_{n,i})\frac{\lambda_{n,i}}{2}\frac{\lambda}{n} +k''(s_{n,i})\frac{\lambda_{n,i}}{3!}\frac{\lambda^2}{n^2}\right].$$
The second term clearly converges to a constant times $\int^\lambda_0 k''(s)ds$ which is $0$. Also the first term converges to $0$, and to see this it is sufficient to rewrite the summation as
\begin{equation} \label{Taylor k''}
\lim_{n\to\infty} \frac{1}{4\cdot 5!}{\lambda^5 n} \left[ \sum_{i=1}^n \int_{s_{n,i}}^{s_{n,i+1}}(k'(s_{n,i})-k'(s))ds\right]
\end{equation}
and then use the Taylor expansion of $k'(s)$ around $s_{n,i}$. Therefore, it remains to compute
\begin{equation} \label{secondo limite}
\lim_{n\to \infty} \frac{n^6}{2\cdot5!} \sum_{i=1}^n-k(s_{n,i})\lambda_{n,i}\left(\lambda_{n,i}^4-\frac{\lambda^4}{n^4}\right) = -\frac{1}{15\cdot5!}\lambda^6\int_0^\lambda k^2(s)ds+\frac{1}{15\cdot5!}\lambda^5\left(\int_0^\lambda k(s)ds\right)^2.
\end{equation}
where, the last equality, is a straightforward application of Proposition \ref{L2}.  Again, we have

        \indent By arguing as in (\ref{Taylor k''}), we obtain that the third term (of order $6$) in the expansion of (\ref{differenza}) does not give contribution:
        $$\lim_{n\to\infty}\frac{n^6}{2}\sum_{i=1}^n-\frac{3}{6!}k'(s_{n,i})\lambda_{n,i}^6= 0.$$

        \indent Finally, we compute the limit of the last term:
        \begin{equation*}
\lim_{n\to\infty} \frac{n^6}{2} \sum_{i=1}^n \left[
-\frac{1}{7\cdot5!}\frac{\lambda^7}{n^7}k''(s_{n,i})+\frac{1}{7!}\frac{\lambda^7}{n^7}k^2(s_{n,i})\right] =\lim_{n\to\infty}\frac{\lambda^6}{2}\sum_{i=1}^n \left[-\frac{1}{7\cdot5!}\lambda_{n,i} k''(s_{n,i})+\frac{1}{7!}\lambda_{n,i} k^2(s_{n,i})
\right] 
\end{equation*}
\begin{equation} \label{ultimo contributo}
=\frac{1}{2\cdot7!}\lambda^6\int_0^\lambda k^2(s)ds.
        \end{equation}
The coefficients of point $(a)$ of the theorem are obtained by summing up (\ref{primo contributo}), (\ref{secondo limite}) and (\ref{ultimo contributo}). From (\ref{differenza}), we immediately obtain that the term of order $5$ has zero coefficient. This fact follows also from S. Marvizi and R. Melrose’s theory, as recalled in Section \ref{MAIN}. \\
\indent 
Point $(b)$ can be proved in the same way. 
\end{proof}
\noindent As pointed out in \cite{ALB}[Theorem 6], in the symplectic billiard case, the first two coefficients $\beta_1$ and $\beta_3$ make it possible to distinguish an ellipse. 
In fact, the inequality
$$ 3 \beta_3 \le - 2 \pi^2 \beta_1$$
equals to the affine isoperimetric inequality
$$\lambda^3 \le 8\pi^2 Area(C),$$
which always holds for every strictly-convex closed curve $\partial C$ and it is an equality only for ellipses. This is not the case of outer billiards, since $\beta_1 = 0$. However, in the next corollary we underline that coefficients $\beta_5$ and $\beta_7$ --even if not geometrically significant-- allow to fix an ellipse, also in the outer billiard case.
\begin{corollary} \label{TAB} Same assumptions of Theorem \ref{TH}. The coefficients $\beta_5$ and $\beta_7$ recognize an ellipse. In particular:
\begin{enumerate}
\item[$(a)$] For symplectic billiards, one always has the inequality 
\begin{equation}
    42\lambda^3\beta_7\leq5!(\beta_5)^2,
\end{equation}with equality if and only if $\partial C$ is an ellipse.
\item[$(b)$] For outer billiards, one always has the inequality 
\begin{equation}
    7\lambda^3\beta_7\geq170(\beta_5)^2,
\end{equation}with equality if and only if $\partial C$ is an ellipse.
\end{enumerate}
\end{corollary}
\begin{proof}
    By Cauchy-Schwartz inequality, it holds
    $$\left(\int_0^\lambda k(s) ds\right)^2\leq\lambda\int_0^\lambda k^2(s)ds,$$
    with equality if and only if $k(s)$ is constant (that is, $\partial C$ is an ellipse). As a consequence, in the symplectic billiard case, we get
    $$\beta_7\leq \left( -\frac{9}{5\cdot7!} + \frac{1}{15 \cdot 5!} \right) \lambda^5 \left(\int_0^\lambda k(s) ds\right)^2 = 
    \frac{\lambda^5}{7!}\left(\int_0^\lambda k(s) ds\right)^2 = \frac{5!}{42\lambda^3}(\beta_5)^2$$
    which gives the result of point $(a)$. Point $(b)$ is obtained analogously. \end{proof}

\section{Ellipses and circles} \label{EC}
In the case of circles and ellipses, all coefficients of the Mather's $\beta$-function --both for symplectic and outer billiards-- can be easily obtained directly. In particular, by the affine equivariance of both maps, it is sufficient to consider the case of circular tables. In this final section, we compute these coefficients for circles (and therefore for ellipses) and check their consistency with the $\beta_i$'s of Theorem \ref{TH}. 

\subsection{Symplectic billiards}
For a disc centered in $0$ with radius $R$, the generating function is twice the area of the triangle $x0y$: 
$$\omega(x,y) = 2 Area(x0y)$$
and periodic orbits of rotation number $\frac{1}{n}$ ($n \ge 3$) correspond to inscribed regular polygons with $n$ edges. \\
\noindent The generating function for an ellipse $\frac{x_1^2}{a^2} + \frac{x_2^2}{b^2} = 1$ is obtained by applying the affine transformation $\begin{pmatrix}
   a/R & 0\\0 & b/R
 \end{pmatrix}$, so that
$$\omega(x,y) = ab\sin\left(\frac{2\pi}{n}\right).$$
As a consequence, from Definition \ref{Mather beta} of Mather's $\beta$-function, it follows


\begin{equation}
\beta \left(\frac{1}{n}\right)= ab
\sum_{k=0}^\infty \frac{(-1)^{k+1}}{(2k+1)!}\left(\frac{2\pi}{n}\right)^{2k+1}.
 \end{equation}
Since the affine length and curvature of the ellipse (and circle, for $a = b = R$) are respectively $\lambda = 2\pi (ab)^{1/3}$ and $k=(ab)^{-2/3}$, the above coefficients are consistent with the ones of point $(a)$ of Theorem \ref{TH}. 

 \subsection{Outer billiards}
For the circular table of boundary $x_1^2 + x_2^2 = R^2$, periodic orbits of rotation number $\frac{1}{n}$ ($n \ge 3$) correspond to circumscribed regular polygons with $n$ edges and the generating function is the area:
$$R^2 \left[ \tan \left( \frac{\pi}{n} \right) - \frac{\pi}{n} \right]$$
of the grey region in Figure \ref{Outcircles}. Therefore, the generating function for an ellipse $\frac{x_1^2}{a^2} + \frac{x_2^2}{b^2} = 1$ results:
$$\omega(x,y) = ab \left[ \tan \left( \frac{\pi}{n} \right) - \frac{\pi}{n} \right]$$ 
\begin{figure}[H]
    \centering
    \includegraphics[scale=0.35]{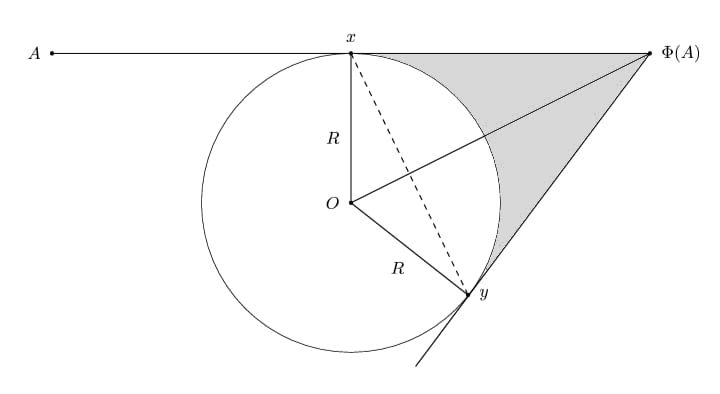}
    \caption{The outer billiard in circles.}
    \label{Outcircles}
\end{figure}
\noindent and the corresponding Mather's $\beta$-function is
$$ab \sum_{k=2}^{\infty} \frac{(-1)^{k-1}4^k(4^k-1)B_{2k}}{(2k)!}\left(\frac{\pi}{n}\right)^{2k-1}$$
(the $B_{2k}$'s are the Bernoulli numbers), whose coefficients are consistent with the ones of point $(b)$ of Theorem \ref{TH}.

\end{document}